\documentclass[preprint,11pt]{elsarticle}

\usepackage{amssymb}
\usepackage{amsmath}
\usepackage{amsfonts}
\usepackage{amssymb}
\usepackage{indentfirst,latexsym,bm}
\usepackage{amsthm}
\usepackage{color,xcolor}
\usepackage[all]{xy}
\input{mathrsfs.sty}
\newtheorem{theorem}{Theorem}[section]
\newtheorem{lemma}[theorem]{Lemma}
\newtheorem{corollary}[theorem]{Corollary}

\theoremstyle{definition}

\theoremstyle{definition}
\newtheorem{example}{Example}[section]

\theoremstyle{remark}
\newtheorem{remark}{Remark}[section]
\theoremstyle{question}

\theoremstyle{problem}

\numberwithin{equation}{section}

\journal{XXX}

\begin{document}

\begin{frontmatter}

\title{Norm attainment of a class of block operator matrices}
\author[shnu]{Kangjian Wu}
\ead{wukjcool@163.com}
\author[shnu]{Jiayu Ling}
\ead{1840206596@qq.com}
\author[shnu]{Qingxiang Xu}
\ead{qingxiang\_xu@126.com}
\address[shnu]{Department of Mathematics, Shanghai Normal University, Shanghai 200234, PR China}

\begin{abstract} Given complex numbers $a, b, c$ and a non-negative continuous function $\varphi$ defined on $[0, +\infty)$, consider the $2 \times 2$ matrix
$$
M_t = \begin{pmatrix} a & t \\ ct & b\varphi(t) \end{pmatrix}, \quad t \in [0, +\infty).
$$
We establish  conditions for the strict monotonicity of the norm function $t \mapsto \|M_t\|$. As an application, we characterize the norm attainment of the corresponding block operator matrix
$$
T = \begin{pmatrix} aI_H & A \\ cA^* & b\varphi(|A|) \end{pmatrix},
$$
where $I_H$ is the identity operator on a Hilbert space $H$ and $A$ is a bounded linear operator from another Hilbert space to $H$.
\end{abstract}

\begin{keyword} Matrix norm function, block operator matrix, norm attainment
\MSC 47A05



\end{keyword}

\end{frontmatter}

\section{Introduction}

Throughout this paper, $\mathbb{N}$, $\mathbb{C}$, $\mathbb{R}$, and $\mathbb{R}_+$ denote the sets of positive integers, complex numbers, real numbers, and non-negative real numbers, respectively.
For Hilbert spaces $H$ and $K$, let $\mathbb{B}(K,H)$ be  the space of bounded linear operators from $K$ to $H$, with $\mathbb{B}(H):=\mathbb{B}(H,H)$. 
The identity operator on $H$ and the cone of positive operators in $\mathbb{B}(H)$ are denoted by $I_H$ and $\mathbb{B}(H)_+$, respectively. 

For an operator $A\in\mathbb{B}(H)$, we write $\mathcal{R}(A)$ and $\mathcal{N}(A)$ for its range and null space, respectively, while $\sigma(A)$ and $\sigma_p(A)$ denote its spectrum and point spectrum.
We write $C\big(\sigma(A)\big)$ for the algebra of continuous complex-valued functions on $\sigma(A)$.
For a square matrix $M$, $\operatorname{tr}(M)$ and $\operatorname{det}(M)$ denote its trace and determinant.
Unless stated otherwise, $\|\cdot\|$ refers to the operator norm.

An operator $A \in \mathbb{B}(K,H)$ is said to attain its norm \cite[Section~2.1]{WG} if $\|Ax\| = \|A\|$ for some unit vector $x \in K$; see \cite{BDSS,CN,PP,Ramesh} for related characterizations.
Given $a, b, c \in \mathbb{C}$ and $\varphi: [0, +\infty) \to \mathbb{R}$, we consider the block operator matrix
$$
T = \begin{pmatrix} aI_H & A \\ cA^* & b\varphi(|A|) \end{pmatrix},
$$
where $A\in\mathbb{B}(K,H)$ and $|A|=(A^*A)^\frac12$. Motivated by the characterization for the closedness of the numerical range of a quadratic operator \cite[Theorem~2.1]{TW},
the norm attainment of $T$ was characterized in \cite{WX} for the constant case $\varphi \equiv 1$. Replacing this special constant function with a general non-negative continuous  function $\varphi$ defined on $\mathbb{R}_+$, this paper aims to provide some non-trivial generalizations of the main results in \cite{WX}, which will be dealt with in Section~\ref{sec:norm attainment}. To this end, in Section~\ref{sec:strict monotonicity} we are led to investigate the strict monotonicity of the corresponding matrix norm functions.

\section{Strict monotonicity of a class of  matrix norm functions}\label{sec:strict monotonicity}
Given $a, b, c \in \mathbb{C}$, and a real-valued function $\varphi$ defined on $\mathbb{R}_+$, let
\begin{equation}\label{defn of M t}M_t=\left(
        \begin{array}{cc}
          a & t \\
          c t & b\varphi(t)\\
        \end{array}
      \right),\quad t\in \mathbb{R}_+.\end{equation}
The main purpose of this section is to investigate conditions under which the function $t\to \|M_t\|$ is
strictly increasing in $t$ on $\mathbb{R}_+$. To this end, we present the following technical result.

\begin{theorem}\label{thm:norm increasing-01}Let $\varphi:\mathbb{R}_+\to \mathbb{R}$ be a function satisfying the following conditions:
\begin{enumerate}
\item[{\rm (i)}] $\varphi$ is continuous on $\mathbb{R}_+$ and differentiable on $(0, +\infty)$;
\item[{\rm (ii)}] $\varphi$ is nonnegative and strictly increasing on $\mathbb{R}_+$.
\end{enumerate}
Then the following statements are equivalent:
\begin{enumerate}
\item[{\rm (a)}] For every $t_0>0$, $\varphi(t) \leq \varphi(t_0)(t/t_0)^4$ for all $t\in [t_0,+\infty)$;
\item[\rm (b)]  $\varphi(t)\ge \frac14 t\varphi^\prime(t)$ for all $t\in (0,+\infty)$;
\item[{\rm (c)}] For any $a, b, c \in \mathbb{C}$,  the function $t\to \|M_t\|$ is
strictly increasing in $t$ on $\mathbb{R}_+$, where $M_t$ is defined in \eqref{defn of M t}
\end{enumerate}
\end{theorem}
\begin{proof}(a) $\Longrightarrow$ (b). Fix an arbitrary $t_0\in (0,+\infty)$. For any $t>t_0$, we have
\begin{align*}
    \frac{\varphi(t) - \varphi(t_0)}{t - t_0} \leq \varphi(t_0) \frac{(t/t_0)^4- 1}{t - t_0}.
\end{align*}
Taking the limit as $t \to t_0^+$ yields
\begin{align*}
    \varphi^\prime(t_0) \leq \varphi(t_0)\cdot \frac{4}{t_0}.
\end{align*}
Hence, $\varphi(t_0) \ge \frac{1}{4} t_0\varphi^\prime(t_0)$, as required.

(b) $\Longrightarrow$ (a). By (ii), we have $\varphi(t)>\varphi(0)\ge 0$ for all $t>0$. It follows that
 $$\frac{\varphi^\prime(t)}{\varphi(t)} \leq \frac{4}{t},\quad \forall t\in (0,+\infty).$$
 Thus, for any $t, t_0\in (0,+\infty)$ with $t>t_0$,
 $$  \ln\varphi(t) - \ln\varphi(t_0)=\int_{t_0}^t \frac{\varphi^\prime(s)}{\varphi(s)}ds\le \int_{t_0}^t \frac{4}{s}ds=4(\ln t - \ln t_0).$$
Exponentiating both sides yields the inequality in (a).

(b)$\Longrightarrow$(c). Notably, by (i) and (ii) we have  $\varphi^\prime(t)\ge 0$ for every $t\in (0,+\infty)$.
First, we consider the case that $bc\ne 0$. Define
$$A_t=M_t^*M_t,\quad f(t)=\|A_t\|,\quad g(t)=\big[\mbox{tr}(A_t)\big]^2-4\mbox{det}(A_t).$$
Since $A_t$ is a positive semi-definite $2\times 2$ matrix, its characteristic equation is
\begin{equation}\label{2 eigenvalues}\lambda^2-\mbox{tr}(A_t)\lambda+\mbox{det}(A_t)=0, \end{equation}
where the eigenvalues satisfy $f(t)=\lambda_1\ge \lambda_2\ge 0$. Hence,
\begin{align}f(t)=&\max\{\lambda_1,\lambda_2\}=\frac{\lambda_1+\lambda_2+\sqrt{(\lambda_1+\lambda_2)^2-4\lambda_1\lambda_2}}{2}\nonumber\\
\label{concrete exp}=&\frac{\mbox{tr}(A_t)+\sqrt{g(t)}}{2}.
\end{align}
Direct computation yields
\begin{align}\label{equ:exp of A t}
   A_t= \begin{pmatrix}
        a_{11}(t) & t\overline{h(t)}\\
        th(t) & a_{22}(t)
    \end{pmatrix},
\end{align}
in which
\begin{align}\label{equ:exp of a11, a22, h}
    a_{11}(t) = |a|^2 + |c|^2 t^2,\  h(t)= a + \bar{b}c\varphi(t), \  a_{22}(t) = t^2 + |b|^2\varphi^2(t).
\end{align}
Consequently,
\begin{align}g(t)=&\big[a_{11}(t)-a_{22}(t)\big]^2+4a_{11}(t)a_{22}(t)-4\mbox{det}(A_t)\notag\\
\label{expression of g}=&\big[a_{11}(t)-a_{22}(t)\big]^2+4t^2|h(t)|^2.
\end{align}
Since $\varphi$ is strictly increasing, the set $E=\big\{t\in (0,+\infty): h(t)=0\big\}$ is either a singleton or empty.
From the above expression for $g(t)$, we have $g(t)>0$ for all $t\in (0,+\infty)\setminus E$. This, together with \eqref{concrete exp}, implies that
   $f$ is differentiable at each  point in $(0,+\infty)\setminus E$. Since $f$ is continuous on $\mathbb{R}_+$ and $E$ is finite, it suffices to show  that
  $f^\prime (t)> 0$ for all $t\in (0,+\infty)\setminus E$.
Notably, for each $t\in (0,+\infty)\setminus E$, we have
$$2f(t)-\operatorname{tr}(A_t)=\sqrt{g(t)}\ge 2t|h(t)|>0.$$
Therefore, it suffices to prove that
\begin{equation}\label{equ:lings tech}f^\prime (t)\big[2f(t)-\operatorname{tr}(A_t)\big]> 0,\quad \forall t\in (0,+\infty)\setminus E.\end{equation}

In what follows, we assume $t\in (0,+\infty)\setminus E$.  Since $f(t)$ is a root of the characteristic equation \eqref{2 eigenvalues} of $A_t$, we have
\begin{align}\label{equ:2 mupliplication}
    \big[f(t) - a_{11}(t)\big] \big[f(t) - a_{22}(t)\big]= t^2|h(t)|^2.
\end{align}
Differentiating both sides with respect to $t$  gives
\begin{align*}
      f^\prime(t)\bigl[2f(t)-\operatorname{tr}(A_t)\bigr]=&a_{11}^\prime(t)\big[f(t) - a_{22}(t)\big]+ 2t|h(t)|^2+t^2 (|h(t)|^2)^\prime  \\
      &\qquad+ a_{22}^\prime(t)\big[f(t) - a_{11}(t)\big]\\
    =&x_1(t)+x_2(t)+x_3(t),
\end{align*}
where
\begin{align*}&x_1(t)=2t|c|^2\bigl[f(t)-a_{22}(t)\bigr],\quad x_2(t)=2t|h(t)|^2 + 2t^2\mathrm{Re}\bigl(b\bar{c}h(t)\bigr) \varphi^\prime(t),\\
&x_3(t)=\bigl[2t+2|b|^2\varphi(t)\varphi^\prime(t)\bigr]\bigl[f(t)-a_{11}(t)\bigr].\end{align*}

Since $A_t$ is given by \eqref{equ:exp of A t} and $f(t)=\|A_t\|$, we have
$$f(t)\ge \max\{a_{11}(t),a_{22}(t)\}.$$
By the AM-GM inequality,
\begin{align*}x_1(t)+x_3(t)\ge 2\sqrt{x_1(t)x_3(t)}=y(t)\sqrt{ \big[f(t) - a_{11}(t)\big] \big[f(t) - a_{22}(t)\big]},
\end{align*}
where
$$y(t)=4|c|\sqrt{t\big(t+|b|^2\varphi(t)\varphi^\prime(t)\big)}.
$$
Substituting \eqref{equ:2 mupliplication} into the right-hand side of the preceding inequality  yields
$$x_1(t)+x_3(t)\ge y(t)t|h(t)|.$$
Thus,
\begin{align*}
    \sum_{i=1}^3x_i(t)&\geq y(t)t|h(t)| + 2t|h(t)|^2 + 2t^2\mathrm{Re}\bigl(b\bar{c}h(t)\bigr) \varphi^\prime(t) \\
    &\geq 4|bc|t^{\frac32}\sqrt{\varphi(t)\varphi^\prime(t)}|h(t)| + 2t|h(t)|^2 - 2t^2|bc|\cdot |h(t)|\varphi^\prime(t) \\
     &> 4|bc|t^{\frac32}\sqrt{\varphi(t)\varphi^\prime(t)}|h(t)| - 2t^2|bc|\cdot |h(t)|\varphi^\prime(t) \\
       &= 4|bch(t)|t^{\frac32}\sqrt{\varphi^\prime(t)}\left[\sqrt{\varphi(t)}-\sqrt{\frac{t\varphi^\prime(t)}{4}}\right]\ge 0.
\end{align*}
This completes the proof of \eqref{equ:lings tech}.

Next, we consider the case $bc=0$. In this case, $h(t)\equiv a$. If $a\ne 0$, then the set $E$ defined earlier is empty, so the desired conclusion follows. If $bc=0$ and $a=0$, then $h(t)\equiv 0$, which implies $\sqrt{g(t)}=|a_{11}(t)-a_{22}(t)|$. Consequently,
$$f(t)=\max\{a_{11}(t),a_{22}(t)\}=\max\big\{|c|^2 t^2, t^2 + |b|^2\varphi^2(t)\big\}.$$
Since the mapping $t\to t^2$ is strictly increasing on $\mathbb{R}_+$, and since, by assumption,  $\varphi^2$ is monotonically increasing on $\mathbb{R}_+$,
it follows that for either $b=0$ or $c=0$, we have
$f(t_1)<f(t_2)$ whenever $0\le t_1<t_2$.

(c)$\Longrightarrow$(b). Suppose there exists some $t_0 \in (0, +\infty)$ such that $$t_0\varphi^\prime(t_0) > 4\varphi(t_0).$$
 Since $t_0>0$  and $\varphi(t_0) \geq 0$, we have  $\varphi^\prime(t_0) > 0$. We now construct explicit parameters $a,b$ and $c$ so that the function $f$ defined above fails to be strictly increasing in $t$ on $\mathbb{R}_+$. For this purpose, we set
 \begin{equation}\label{construction of abc}c=1, \quad b=\sqrt{\frac{4t_0+1}{\varphi^\prime(t_0)\big[t_0\varphi^\prime(t_0)-4\varphi(t_0)\big]}}, \quad a = -b\varphi(t_0).\end{equation}
 Then
 \begin{align*}h(t)=b\big[\varphi(t)-\varphi(t_0)\big],\quad a_{11}(t)=b^2\varphi^2(t_0)+t^2,\quad a_{22}(t) = t^2 + b^2\varphi^2(t).
 \end{align*}
 From \eqref{concrete exp}--\eqref{expression of g}, it follows that
 \begin{align}\label{new exp g}&g(t)=b^4\big[\varphi^2(t)-\varphi^2(t_0)\big]^2+4t^2b^2\big[\varphi(t)-\varphi(t_0)\big]^2,\\
 \label{new exp f}&f(t)=\frac12\left[2t^2+b^2\big[\varphi^2(t)+\varphi^2(t_0)\big]+\sqrt{g(t)}\right].
 \end{align}
Hence,
$$g(t_0)=0,\quad f(t_0)=t_0^2 + b^2\varphi^2(t_0).$$
Let $\Delta t=t-t_0$. Then
$$ \varphi(t)=\varphi(t_0)+\varphi^\prime(t_0)\Delta t+o(\Delta t).$$
Substituting this expansion of $\varphi(t)$ into \eqref{new exp g} and \eqref{new exp f}  gives
\begin{align*}g(t)=&4b^4\big[\varphi(t_0)\varphi^\prime(t_0)\big]^2 (\Delta t)^2+4t_0^2b^2\big(\varphi^\prime(t_0)\big)^2 (\Delta t)^2+o\big((\Delta t)^2\big)\\
=&4b^2\big(\varphi^\prime(t_0)\big)^2 f(t_0)(\Delta t)^2+o\big((\Delta t)^2\big),\\
f(t)=&(t_0+\Delta t)^2+b^2\left[\varphi^2(t_0)+\varphi(t_0)\varphi^\prime(t_0)\Delta t+o(\Delta t)\right]+\frac12 \sqrt{g(t)}\\
=&f(t_0)+\left[2t_0+b^2\varphi(t_0)\varphi^\prime(t_0)\right]\Delta t+b\varphi^\prime(t_0)\sqrt{f(t_0)}|\Delta t|+o(\Delta t).
\end{align*}
Therefore, for $-t_0<\Delta t<0$,
$$f(t_0+\Delta t)-f(t_0)=d|\Delta t|+\alpha(\Delta t),$$
where
 $d=d_1-d_2$, with
$$ d_1 = b\varphi^\prime(t_0)\sqrt{t_0^2 + b^2\varphi^2(t_0)} \quad \text{and} \quad d_2 = 2t_0+b^2\varphi(t_0)\varphi^\prime(t_0). $$
From the definition of $b$ in \eqref{construction of abc}, we calculate
\begin{align*}
    d_1^2 - d_2^2 &= t_0b^2 \varphi^\prime(t_0) \big[t_0\varphi^\prime(t_0) - 4\varphi(t_0)\big] - 4t_0^2 \\
    &= t_0(4t_0+1) - 4t_0^2 = t_0 > 0.
\end{align*}
Since $d_1>0$ and $d_2>0$, it follows that $d=d_1-d_2> 0$.

Choose $\delta_0\in (0,t_0)$ such that for $0<|\Delta t|\le \delta_0$,
$$|\alpha(\Delta t)|<\frac{d}{2}|\Delta t|.$$
Then
$$f(t_0-\delta_0)-f(t_0)=d\delta_0+\alpha(-\delta_0)>\frac12 d\delta_0>0.$$
Thus, $f$ is not strictly increasing in $t$ on $\mathbb{R}_+$.
\end{proof}

\begin{remark}From the proof of the preceding theorem, we observe that the number $4t_0+1$
appearing in \eqref{construction of abc} can be replaced by any number greater than $4t_0$.
\end{remark}

Next,  we draw upon a lemma from \cite{WX}.

\begin{lemma}\label{WX d zero}{\rm\cite[Lemma~2.3]{WX}}
Let $S\in \mathbb{B}(H\oplus K)$ be given by
$$
S =\left(
     \begin{array}{cc}
      a I_H & A \\
c A^* & bI_K\\
       \end{array}
   \right)
,$$  where  $A \in \mathbb{B}(K,H)$ and $a,b,c \in \mathbb{C}$. Then
\begin{equation*}\|S\|^2 = \frac{1}{2}\left(r + \sqrt{(|a|^2 - |b|^2)^2 + (|c|^2 - 1)^2\|A\|^4+2k\|A\|^2}\right),\end{equation*}
where
$$r=|a|^2+|b|^2+|c|^2+1,\quad k=|b+\bar{a}c|^2+|a+\bar{b}c|^2.$$
\end{lemma}

As a consequence of Theorem~\ref{thm:norm increasing-01} and Lemma~\ref{WX d zero}, we obtain the following two corollaries.

\begin{corollary}\label{cor:norm increasing-01} Let $\varphi(t)=k+dt^\alpha$ with $k,d,\alpha\in \mathbb{R}_+$.
Then the following statements are equivalent:
\begin{enumerate}
\item[{\rm (a)}] $\alpha\in [0,4]$;
\item[{\rm (b)}] For any $a, b, c \in \mathbb{C}$,  the function $t\to \|M_t\|$ is
strictly increasing in $t$ on $\mathbb{R}_+$, where $M_t$ is defined in
\eqref{defn of M t}.
\end{enumerate}
\end{corollary}
\begin{proof}If $d=0$ or $\alpha=0$, then for any $a, b, c \in \mathbb{C}$ and $t\in \mathbb{R}_+$,
\begin{equation*}M_t=\left(
        \begin{array}{cc}
          a & t \\
          c t & b_1\\
        \end{array}
      \right),\end{equation*}
      where $b_1$ is a fixed complex number. By Lemma~\ref{WX d zero}, $\|M_t\|$ is
strictly increasing in $t$ on $\mathbb{R}_+$.

Now suppose that $d,\alpha\in (0,+\infty)$. In this case,  conditions (i)--(ii) of Theorem~\ref{thm:norm increasing-01} are clearly satisfied.  Moreover, we have $t \varphi^\prime(t)=\alpha d t^\alpha$ for all $t\in (0,+\infty)$.
Hence,  for every $t\in (0,+\infty)$,
\begin{align*}\varphi(t)\ge \frac14 t\varphi^\prime(t)\Longleftrightarrow  k t^{-\alpha}+d\ge \frac14 \alpha d. \end{align*}
It is clear that  the latter inequality holds for all $t\in (0,+\infty)$ if and only if $\alpha\le 4$.
\end{proof}

\begin{example}\label{ex:01}
Let $\varphi(t)=2t^5$ and $t_0=1$. Then
\[
\varphi(t_0)=2,\qquad
\varphi'(t_0)=10,\qquad
t_0\varphi'(t_0)-4\varphi(t_0)=2.
\]
Guided by the construction in the proof of Theorem~\ref{thm:norm increasing-01}, we choose $c=1$ and define
\[
b=\sqrt{\frac{20t_0}{\varphi'(t_0)\bigl[t_0\varphi'(t_0)-4\varphi(t_0)\bigr]}}=1,
\]
and set $a=-b\varphi(t_0)=-2$. It follows that for $t>0$,
\begin{equation}\label{M t 4 counterexample}
M_t=
\begin{pmatrix}
-2 & t \\
t  & 2t^5
\end{pmatrix},
\end{equation}
whose eigenvalues are
\[
\lambda_{1,2}=t^5-1\pm\sqrt{t^{10}+2t^5+t^2+1}.
\]
Since $M_t$ is a real symmetric matrix, we have
\[
f(t):=\|M_t\|
=\max\{|\lambda_1|,|\lambda_2|\}
=1-t^5+\sqrt{t^{10}+2t^5+t^2+1},
\qquad t\in[0,1].
\]

A direct computation shows that for $t\in(0,1)$,
\[
f'(t)=\frac{t\,g(t)}{\sqrt{t^{10}+2t^5+t^2+1}},
\]
where
\[
g(t)=5t^8+5t^3+1-5t^3\sqrt{t^{10}+2t^5+t^2+1}.
\]

One verifies that for $t\in[0,1]$,
\[
g(t)<0\iff h(t)>0
\quad\text{and}\quad
g(t)=0\iff h(t)=0,
\]
with
\begin{equation}\label{defn of h}
h(t)=15t^8-10t^3-1.
\end{equation}

Let $t_1=4^{-1/5}$. Then
\[
h(t_1)=15t_1^8-10t_1^3-1
=-\frac{25}{4}t_1^3-1<0,
\]
so that $h(t_1)h(1)<0$. Hence there exists $t_*\in(t_1,1)$ such that $h(t_*)=0$.

Since $h'(t) = 30t^2(4t^5 - 1)$, the factor $4t^5 - 1$ dictates the sign of the derivative. Consequently, $h$ is strictly decreasing on $[0, t_1]$ and strictly increasing on $[t_1, 1]$. Given that $h(0)=-1$ and  $h(t_*) = 0$, it follows immediately that $h(t) <\max\{h(0), h(t_*)\}=0$ for $t \in (0, t_*)$, and $h(t) > 0$ for $t \in (t_*, 1]$. By the equivalence established earlier, this implies $g(t_*) = 0$, with $g(t) > 0$ on $(0, t_*)$ and $g(t) < 0$ on $(t_*, 1]$. Therefore, $\|M_t\|$ is strictly increasing on $[0, t_*]$ and strictly decreasing on $[t_*, 1]$.

Numerical computation (MATLAB) gives $t_*\approx0.9431$, and
\[
\|M_{0.9431}\|\approx2.2384
\;>\;
2.2361\approx\|M_1\|.
\]
Thus, the mapping $t\mapsto\|M_t\|$ is \emph{not} strictly increasing on $\mathbb{R}_+$.
\end{example}

\begin{corollary}For any $\alpha>0$, define $\varphi(t)=\ln(1+\alpha t)$ for $t>0$. Then for all $a, b, c \in \mathbb{C}$,  the function $t\to \|M_t\|$ is
strictly increasing in $t$ on $\mathbb{R}_+$, where $M_t$ is given by \eqref{defn of M t}.
\end{corollary}
\begin{proof}For any $t>0$, we compute
\begin{align*}t\varphi^\prime(t)=\frac{\alpha t}{1+\alpha t}\le \varphi(t).
\end{align*}
The result follows directly from Theorem~\ref{thm:norm increasing-01}.
\end{proof}

To proceed, we require another technical result as follows.

\begin{theorem}\label{thm:norm increasing iff condition}
    Let $a,c,d \in \mathbb{C}$. For $s\in [0,+\infty)$, define $N_s\in M_2(\mathbb{C})$ by
    \begin{equation*}
        N_s=\begin{pmatrix}
            a & d \\
            c & s
        \end{pmatrix}.
    \end{equation*}
    Then $\|N_s\|$ is a strictly increasing function of $s$ on $[0,+\infty)$ if and only if one of the following conditions holds:
    \begin{enumerate}
        \item[\upshape (i)] $\mathrm{Re}\bigl(\bar{a}cd\bigr) > 0$;
        \item[\upshape (ii)] $\mathrm{Re}\bigl(\bar{a}cd\bigr) = 0$ and $|c| + |d| > 0$;
        \item[\upshape (iii)] $a = c = d = 0$.
    \end{enumerate}
\end{theorem}

\begin{proof}For $s\in [0,+\infty)$, let $B_s=N_s^*N_s$. Direct computation gives
\begin{align*}
B_s=\left(
      \begin{array}{cc}
        |a|^2+|c|^2 & \bar{a}d+\bar{c}s \\
        a\bar{d}+cs & |d|^2+s^2 \\
      \end{array}
    \right).
\end{align*}
As shown in the proof of Theorem~\ref{thm:norm increasing-01},  we have
\begin{equation*}\|N_s\|^2=\|B_s\|=\frac{\mbox{tr}(B_s)+\sqrt{f(s)}}{2},
\end{equation*}
where
\begin{align*}f(s)=&\big[\mbox{tr}(B_s)\big]^2-4\mbox{det}(B_s)=\big(|a|^2+|c|^2+|d|^2+s^2\big)^2-4\mbox{det}(B_s)\\
=&\big(|d|^2+s^2-|a|^2-|c|^2\big)^2+4(|d|^2+s^2)(|a|^2+|c|^2)-4\mbox{det}(B_s)\\
=&\left(s^2+|d|^2-|a|^2-|c|^2\right)^2+4\left|a\bar{d}+cs\right|^2.
\end{align*}
It follows that  the set $E=\big\{s\in (0,+\infty): f(s)=0\big\}$ is either a singleton or empty, and
\begin{align}\label{half f}&\sqrt{f(s)}\ge \Big|s^2+|d|^2-|a|^2-|c|^2\Big|,\\
&f^\prime(s)=4s\big(s^2+|d|^2-|a|^2-|c|^2\big)+8\big(|c|^2s+\mbox{Re}(\bar{a}cd)\big).\nonumber
\end{align}

To analyze the monotonicity of $\|N_s\|$, define
$$g(s)=s^2+\sqrt{f(s)},\quad s\in [0,+\infty).$$
Notably, $\|N_s\|^2$ is  strictly increasing if and only if $g$ is  strictly increasing.
 A direct  calculation shows that for all $ s\in (0,+\infty)\setminus E$,
\begin{equation*}
    g^\prime(s)=
        2s+\frac{2s\big(s^2+|d|^2-|a|^2-|c|^2\big)+4\big(|c|^2s+\mbox{Re}(\bar{a}cd)\big)}{\sqrt{f(s)}}=\frac{2 h(s)}{\sqrt{f(s)}},\end{equation*}
where
\begin{equation}\label{3 terms of h}h(s)=h_1(s)+h_2(s)+2\mathrm{Re}(\bar{a}cd),\end{equation}
in which
$$h_1(s)=s\Big(\sqrt{f(s)}+s^2+|d|^2-|a|^2-|c|^2\Big),\quad h_2(s)=2|c|^2 s.$$

 Suppose  $\|N_s\|$ is strictly increasing on $[0, +\infty)$. Then $g^\prime(s)\ge 0$ for all $ s\in (0,+\infty)\setminus E$, implying $h(s)\ge 0$ for such $s$. Hence,
$$2\mathrm{Re}(\bar{a}cd)=\lim_{s\to 0^+} h(s)\ge 0.$$
If none of (i)--(iii) holds, then $a\ne 0$ and $c=d=0$. In this case, $\|N_s\| = \max\{|a|, s\}$, which is not strictly increasing (e.g.,\,$\|N_{0}\|=\|N_{|a|}\|$). This contradiction shows one of (i)--(iii) must hold.

 Conversely, assume one of (i)--(iii) holds.
    By \eqref{half f}, we have $h_1(s)\ge 0$ for all $s\in (0,+\infty)$. So if either $\mathrm{Re}(\bar{a}cd)>0$, or $c\ne 0$ with $\mathrm{Re}(\bar{a}cd)=0$,
    then using \eqref{3 terms of h}
    we conclude that $h(s)>0$ for all $s\in \mathbb{R}_+$. It follows that $g^\prime(s)>0$ for any $ s\in (0,+\infty)\setminus E$, and thus $\|N_s\|$ is strictly increasing on $[0, +\infty)$.

    For each $s\in \mathbb{R}_+$, let $N_s^T$ denote the transpose of $N_s$. Since $\|N_s\|=\|N_s^T\|$, it follows that $\|N_s\|$ is strictly increasing on $[0, +\infty)$ whenever $d\ne 0$ and $\mathrm{Re}(\bar{a}cd)=0$.

    Finally, if $a=c=d=0$, then $\|N_s\|=s$, so $\|N_s\|$ is obviously strictly increasing on $[0, +\infty)$.
    This completes the proof.
\end{proof}

A combination of Lemma~\ref{WX d zero} and Theorem~\ref{thm:norm increasing iff condition} yields the following corollary.

\begin{corollary}\label{cor:sufficient condition increasing norm}Let $M_t$ be defined by \eqref{defn of M t} for $t\in\mathbb{R}_+$,
where $a,b,c\in\mathbb{C}$ satisfy $\mathrm{Re}\bigl(\bar{a}\bar{b}c\bigr)\ge 0$. If $\varphi:\mathbb{R}_+\to \mathbb{R}$ is nonnegative and  strictly increasing,
then $\|M_t\|$ is strictly increasing in $t$ on $\mathbb{R}_+$.
\end{corollary}
\begin{proof}If $b=0$, then by Lemma~\ref{WX d zero}, $\|M_t\|$ is strictly increasing on $\mathbb{R}_+$.

Now suppose  $b\ne 0$.  Write $b=|b|e^{\mathrm{i}\theta}$ for some $\theta\in\mathbb{R}$, and define $c_1=ce^{-i\theta}$. Consider the  matrices
    $$U_\theta=\left(
                 \begin{array}{cc}
                   1 & 0 \\
                   0 & e^{-i\theta} \\
                 \end{array}
               \right),\quad B_t=\begin{pmatrix}
a & t\\
c_1t & |b|\,\varphi(t)
\end{pmatrix},\quad t\ge 0.$$
For any $t\in [0,+\infty)$, we have $B_t=U_\theta
    M_t$, so $\|M_t\|=\|B_t\|$.  Observe that for any $t\in (0,+\infty)$,
    $$\mbox{Re}(\bar{a}tc_1t)=\frac{t^2}{|b|}\mbox{Re}(\bar{a}\bar{b}c)\ge 0\quad \mbox{and}\quad t+|c_1|t>0.$$
 By Theorem~\ref{thm:norm increasing iff condition},   for any $t>0$ and $s_2>s_1\ge 0$,
$$\left\|
    \begin{pmatrix}
a & t\\
c_1t & s_1\end{pmatrix}
\right\|< \left\|
    \begin{pmatrix}
a & t\\
c_1t & s_2
\end{pmatrix}
\right\|.$$
 Meanwhile, Lemma~\ref{WX d zero} implies that for any $z\in\mathbb{C}$ and $s_2>s_1\ge 0$,
$$\left\|
    \begin{pmatrix}
a & s_1\\
c_1s_1 & z\end{pmatrix}
\right\|< \left\|
    \begin{pmatrix}
a & s_2\\
c_1s_2 & z
\end{pmatrix}
\right\|.$$
Consequently, for any $t_1,t_2\in (0,+\infty)$ with $t_1<t_2$, we have
    \begin{align*}\|B_0\|\le \left\|
    \begin{pmatrix}
a & 0\\
0 & |b|\,\varphi(t_1)
\end{pmatrix}
\right\|<\|B_{t_1}\|
<\left\|
    \begin{pmatrix}
a & t_1\\
c_1t_1 & |b|\,\varphi(t_2)
\end{pmatrix}
\right\|
<\|B_{t_2}\|.
\end{align*}
Therefore, $\|M_{t_2}\|>\|M_{t_1}\|>\|M_0\|$. This demonstrates that  $\|M_t\|$ is strictly increasing in $t$ on $\mathbb{R}_+$.
\end{proof}

\section{Norm attainment of a class of block operator matrices}\label{sec:norm attainment}
In this section, we investigate the
  norm attainment  of the block operator matrix $T$ defined in \eqref{def of T}. The central goal is to precisely characterize the norm attainment of $T$ in terms of its component operator $A$. Throughout this section, $H$ and $K$ are non-trivial Hilbert spaces.

 We begin by recalling  several well-known properties about the norm attainment (see, e.g., \cite[Corollary~2.4, Proposition~2.5]{CN}).

\begin{lemma}\label{lem:attainment-observation}
    For every $A \in \mathbb{B}(K,H)$, the following statements are equivalent:
    \begin{enumerate}
    \item[\upshape (i)] $A$ attains its norm;
    \item[\upshape (ii)] $A^*$ attains its norm;
    \item[\upshape (iii)] $\|A\|^\alpha\in\sigma_p(|A^*|^\alpha)$ for every $\alpha > 0$;
    \item[\upshape (iv)] $\|A\|^\alpha\in\sigma_p(|A^*|^\alpha)$ for some $\alpha > 0$;
    \item[\upshape (v)]  $\|A\|^\alpha\in\sigma_p(|A|^\alpha)$ for every $\alpha > 0$;
     \item[\upshape (vi)]  $\|A\|^\alpha\in\sigma_p(|A|^\alpha)$ for some $\alpha > 0$.
    \end{enumerate}
\end{lemma}

The following lemma establishes the coincidence of the norms of certain operators and the equivalence of their attainment.

\begin{lemma}\label{lem:norm_equivalence}
    Let $\varphi:\mathbb{R}_+\to \mathbb{R}$ be continuous and satisfy $\varphi(t) \geq \varphi(0) \geq 0$ for all $t\in\mathbb{R}_+$.  Then
    $\|T\| = \|\widetilde{T}\|$ for any $A \in \mathbb{B}(K, H)$ and $a,b, c \in\mathbb{C}$,  where $T\in\mathbb{B}(H\oplus K)$ and $\widetilde{T}\in\mathbb{B}(H\oplus H)$ are defined by
    \begin{equation}\label{def of T}
        T = \begin{pmatrix}
            aI_{H} & A \\
            cA^* & b\varphi(|A|)
        \end{pmatrix}, \quad
        \widetilde{T} = \begin{pmatrix}
            aI_{H} & |A^*| \\
            c|A^*| & b\varphi(|A^*|)
        \end{pmatrix}.
    \end{equation}
    Furthermore,
    $T$ attains its norm if and only if $\widetilde{T}$ attains its norm.
    \end{lemma}
\begin{proof}If $A=0$, then the operators $T$ and $\widetilde{T}$ defined above reduce to the diagonal operators $\mbox{diag}(a I_H, b\varphi(0)I_H)$  and $\mbox{diag}(a I_H, b\varphi(0)I_K)$, respectively. Since $H$ and $K$ are nontrivial, the result follows immediately.

Now assume $A \neq 0$.   Define $g(t) = \varphi(t) - \varphi(0)$ for $t\in\mathbb{R}_+$. By hypothesis, $g(t) \geq 0$ for all $t \in \mathbb{R}_+$ and $g(0) = 0$. Consequently,
    \begin{equation*}
        b\varphi(|A|) = b\big[\varphi(0)I_{K} + g(|A|)\big] \quad \text{and} \quad
        b\varphi(|A^*|) = b\big[\varphi(0)I_{H} + g(|A^*|)\big].
    \end{equation*}
    Let $A = V|A|$ be the polar decomposition of $A$. Define the partial isometry $U = \operatorname{diag}(I_{H}, V)\in\mathbb{B}(H\oplus K,H\oplus H)$, and consider the operators
    \begin{equation*}
        X = U T U^*\in\mathbb{B}(H\oplus H) \quad \text{and} \quad Y = U^* X U\in\mathbb{B}(H\oplus K).
    \end{equation*}
    Since $g$ is continuous and vanishes at $0$, we have
    $V g(|A|) V^* = g(|A^*|)$ and $V^* g(|A^*|) V = g(|A|)$.  It follows that
    \begin{align*}
        &X= \begin{pmatrix} aI_{H} & |A^*| \\ c|A^*| & X_{22} \end{pmatrix}, \quad Y=\begin{pmatrix} aI_{H} & A \\ cA^* & Y_{22}\end{pmatrix}=T-Z,
                    \end{align*}
    where
    \begin{align*}&X_{22}=b\big[\varphi(0)VV^*+g(|A^*|)\big],\quad Y_{22}=b\big[\varphi(0)V^*V+g(|A|)\big],\\
    &Z=\begin{pmatrix} 0 & 0\\ 0 & b\varphi(0)(I_K-V^*V)\end{pmatrix}.
    \end{align*}

Observe that
$Y^*Z=0$ and  $YZ^*=0$. Therefore,
$$\|T\|^2=\|(Y-Z)^*(Y-Z)\|=\|Y^*Y+Z^*Z\|=\max\{\|Y\|^2,\|Z\|^2\},$$
which implies
\begin{equation}\label{norms T Y Z}\|T\|=\max\{\|Y\|,\|Z\|\}.
\end{equation}
Since $\varphi(0)V^*V$ and $g(|A|)$ are positive operators, we estimate
    \begin{align*}
        \|Y\| &\geq \|Y_{22}\|=|b|\cdot\big\|\varphi(0)V^*V + g(|A|)\big\|\\
         &\geq |b|\cdot \big\|\varphi(0) V^*V\big\|= |b\varphi(0)|\ge \|Z\|.
    \end{align*}
Hence, by \eqref{norms T Y Z} we obtain $\|T\|=\|Y\|$. Similar reasoning yields $\|\widetilde{T}\|=\|X\|$.
In addition, as $U$ is a contraction, the definitions of $X$ and $Y$ imply
$$\|T\|=\|Y\|\le \|X\|\le \|T\|.
$$
Thus, the norms of $T, X,Y$ and $\widetilde{T}$ are all equal.

    We now turn to the equivalence of norm attainment.  First, we derive a strict inequality. Observe that the lower-right block of $Y^*Y$ is given by $A^*A + Y_{22}^*Y_{22}$. Since $A \neq 0$, we have $\|V^*V\|=1$, and $V^*V$ acts as the unit within the commutative $C^*$-algebra generated by $A^*A$ and $V^*V$. Consequently,
    $$\big\|A^*A+|b\varphi(0)V^*V|^2\big\|=\|A\|^2+|b|^2\varphi^2(0)>|b|^2\varphi^2(0).$$
  It follows that
     \begin{equation*}
        \|Y\|^2 \geq \|A^*A + Y_{22}^*Y_{22}\|>|b|^2 \varphi^2(0)\geq \|Z\|^2,
    \end{equation*}
    establishing the strict inequality $\|Y\| > \|Z\|$.

Next, we show that $T$ attains its norm if and only if $Y$ attains its norm. For brevity, let
$P=U^*U$ and write $I=I_{H\oplus K}$. Clearly,
$YP=Y=PY$ and $ZP=PZ=0$.
So, for any $\xi\in H\oplus K$ with $(I-P)\xi\ne 0$,
\begin{align*}\|T\xi\|^2=&\|(Y+Z)\xi\|^2=\|YP\xi\|^2+\|Z(I-P)\xi\|^2\\
\le& \|Y\|^2\cdot \|P\xi\|^2+\| Z\|^2\cdot \|(I-P)\xi\|^2\\
<&\|Y\|^2\cdot \|P\xi\|^2+\|Y\|^2\cdot \|(I-P)\xi\|^2=\|T\|^2\cdot \|\xi\|^2.
\end{align*}
Hence, if $T$ attains its norm at a unit vector $x$, then $x=Px$ and thus
$\|Yx\|=\|Tx\|=\|T\|=\|Y\|$, so $Y$ also attains its norm at $x$.
Conversely, suppose $x$ is unit vector such that $\|Yx\|=\|Y\|$. Then
$$\|Y\|=\|Yx\|=\|YPx\|\le \|Y\|\cdot \|Px\|\le \|Y\|\cdot \|x\|=\|Y\|.$$
This forces $\|Px\|=\|x\|$, implying $x=Px$. Consequently, $\|Tx\|=\|Yx\|=\|Y\|=\|T\|$, showing that $T$ attains its norm.
Analogously, we conclude that $\widetilde{T}$ attains its norm if and only if $X$ attains its norm.

Finally, we establish the desired equivalence. Suppose $\widetilde{T}$ attains its norm. Then there exists a unit vector $\xi$  such that $\|X\xi\|=\|X\|=\|T\|$. Consequently,
\begin{align*}\|T\|=\|X\xi\|\le \|TU^*\xi\|\le \|T\|\cdot \|U^*\xi\|\le \|T\|.
\end{align*}
This forces $\|U^*\xi\|=1$ and $\|TU^*\xi\|=\|T\|$; hence $T$ attains its norm. Similarly, from the expression of $Y=U^* X U$, we conclude the chain
of implications: $T$ attains its norm $\Longrightarrow$ $Y$ attains its norm $\Longrightarrow$ $X$ attains its norm $\Longrightarrow$ $\widetilde{T}$ attains its norm. This completes the proof.
\end{proof}

\begin{remark}
    Let $T$ and $\widetilde{T}$ be as defined in \eqref{def of T}.
    Observe that $\widetilde{T}$ belongs to the matrix algebra $M_2(\mathfrak{B})$, where
    $\mathfrak{B}$ denotes the $C^*$-subalgebra of $\mathbb{B}(H)$ generated by $I_H$ and $|A^*|$.
    Since $M_2(\mathbb{C})$ is nuclear \cite[Theorem~6.3.9 and 6.4.15]{Murphy} and
    $M_2(\mathfrak{B}) \cong M_2(\mathbb{C}) \otimes \mathfrak{B}$, Lemma~\ref{lem:norm_equivalence} yields
    \begin{equation}\label{equ:norm of T tilde}
        \|T\| = \|\widetilde{T}\| = \max_{t \in \sigma(|A^*|)} \|M_t\|,
    \end{equation}
    where $M_t$ is as defined in \eqref{defn of M t}. We then define
    \begin{equation}\label{defn of Omega}
        \Omega = \bigl\{ t \in \sigma(|A^*|) : \|M_t\| = \|T\| \bigr\}.
    \end{equation}
\end{remark}

To proceed, we establish a lemma that is also of independent interest.

\begin{lemma}\label{lem:kernel_equality}
    Let $A \in \mathbb{B}(H)$ be self-adjoint.
    Suppose $t_0 \in \sigma(A)$ and $\varphi \in C\big(\sigma(A)\big)$ satisfies
    $\mathcal{N}(\varphi) = \{t_0\}$.
    Then $\mathcal{N}\big(\varphi(A)\big) = \mathcal{N}(A - t_0 I_H)$.
\end{lemma}

\begin{proof}
    Let $\varphi_1, \varphi_2 \in C\big(\sigma(A)\big)$ be such that
    $\mathcal{N}(\varphi_1) = \mathcal{N}(\varphi_2) = \{t_0\}$.
    Let $\mathcal{J}$ be the closed two-sided ideal of $C\big(\sigma(A)\big)$ generated by $\varphi_1$, i.e.,
    $$
    \mathcal{J} = \overline{\{u \varphi_1 : u \in C\big(\sigma(A)\big)\}}.
    $$
    Since $\sigma(A)$ is a compact Hausdorff space, by \cite[Theorem~3.4.1]{KR} there exists a unique
    closed subset $F \subseteq \sigma(A)$ such that
    $$
    \mathcal{J} = \big\{h \in C\big(\sigma(A)\big) : h(t) = 0 \text{ for all } t \in F\big\}.
    $$
    As $\varphi_1 \in \mathcal{J}$ and $\mathcal{N}(\varphi_1) = \{t_0\}$, it follows that $F = \{t_0\}$.
    Thus, $\varphi_2 \in \mathcal{J}$. Consequently, there exists a sequence $\{u_n\}_{n=1}^\infty$ in $C\big(\sigma(A)\big)$ such that
    $u_n \varphi_1 \to \varphi_2$ uniformly on $\sigma(A)$. Therefore, for any $x \in \mathcal{N}\big(\varphi_1(A)\big)$,
    \begin{equation*}
        \varphi_2(A)x = \lim_{n\to\infty} u_n(A)\varphi_1(A)x = 0,
    \end{equation*}
    yielding $\mathcal{N}\big(\varphi_1(A)\big) \subseteq \mathcal{N}\big(\varphi_2(A)\big)$.
    The reverse inclusion follows symmetrically, establishing
    $\mathcal{N}\big(\varphi_1(A)\big) = \mathcal{N}\big(\varphi_2(A)\big)$.

    The result follows immediately by letting $\varphi_1 = \varphi$ and $\varphi_2(t) = t - t_0$ for $t \in \sigma(A)$.
\end{proof}

We next prove two lemmas concerning the connection between the norm attainment of $T$ and the point spectrum of $|A^*|$.

\begin{lemma}\label{lem:norm_attainment_necessity}
    Let $\varphi \colon \mathbb{R}_+ \to \mathbb{R}$ be continuous and satisfy $\varphi(t) \geq \varphi(0) \geq 0$ for all $t \in \mathbb{R}_+$.
    Given $A \in \mathbb{B}(K, H)$ and $a, b, c \in \mathbb{C}$, let $M_t$ and $T$ be as defined in \eqref{defn of M t} and \eqref{def of T}, respectively.
    Suppose further that the set $\Omega$ defined by \eqref{defn of Omega} is a singleton $\{t_0\}$ and that $T$ attains its norm.
    Then $t_0 \in \sigma_p(|A^*|)$.
\end{lemma}

\begin{proof}
    If $A = 0$, then $\Omega = \{0\}$ and $0 \in \sigma_p(|A^*|)$, so the conclusion holds.
    Henceforth, assume $A \neq 0$.

    By Lemma~\ref{lem:norm_equivalence}, $T$ attains its norm if and only if $\widetilde{T}$ does.
    Since $\widetilde{T}$ depends only on the positive operator $|A^*|$, Lemma~\ref{lem:attainment-observation} allows us to assume without loss of generality that $H = K$ and $A \geq 0$.

    Assume $T$ attains its norm and let $\lambda_0 = \|T\|^2$.
    By Lemma~\ref{lem:attainment-observation}, there exists a non-zero vector $(u, v)^\top \in H \oplus H$ such that
    \begin{align*}
        T^*T \begin{pmatrix} u \\ v \end{pmatrix} = \lambda_0 \begin{pmatrix} u \\ v \end{pmatrix}.
    \end{align*}
    Writing this out explicitly using the block matrix form of $T^*T$ (with coefficients defined in \eqref{equ:exp of a11, a22, h}) yields the system:
    \begin{equation}\label{equ:T_attaining_system}
        \begin{cases}
            \bigl[\lambda_0 I - a_{11}(A)\bigr]u = A\bigl(h(A)\bigr)^*v, \\[4pt]
            \bigl[\lambda_0 I - a_{22}(A)\bigr]v = Ah(A)u.
        \end{cases}
    \end{equation}

    To eliminate $v$, multiply the first equation on the left by $\lambda_0 I - a_{22}(A)$:
    \begin{align*}
        \bigl[\lambda_0 I - a_{22}(A)\bigr]\bigl[\lambda_0 I - a_{11}(A)\bigr]u
        &= \bigl[\lambda_0 I - a_{22}(A)\bigr]A\bigl(h(A)\bigr)^*v \\
        &= A\bigl(h(A)\bigr)^*\bigl[\lambda_0 I - a_{22}(A)\bigr]v.
    \end{align*}
    Substituting the second equation of \eqref{equ:T_attaining_system} into the right-hand side gives
    \begin{align*}
        \bigl[\lambda_0 I - a_{22}(A)\bigr]\bigl[\lambda_0 I - a_{11}(A)\bigr]u
        &= A\bigl(h(A)\bigr)^* Ah(A)u \\
        &= A^2 |h(A)|^2 u.
    \end{align*}

    Rearranging terms, we find that $u$ lies in the null space  of $\Phi(A)$, where $\Phi \colon \mathbb{R}_+ \to \mathbb{R}$ is the continuous function defined by $$\Phi(t)=\bigl[\lambda_0 - a_{11}(t)\bigr]\bigl[\lambda_0 - a_{22}(t)\bigr] - t^2|h(t)|^2,$$ which can be simplified as
    \begin{equation}\label{defn of Phi}
        \Phi(t)= \det(\lambda_0 I_2 - M_t^*M_t).
    \end{equation}
    Symmetric reasoning applied to the second equation shows that $\Phi(A)v = 0$.

    Now, consider any $t \in \sigma(A)$. If $\Phi(t) = 0$, then $\lambda_0$ is an eigenvalue of $M_t^*M_t$ and thus $\lambda_0 \le \|M_t^*M_t\| = \|M_t\|^2$. Combining this with \eqref{equ:norm of T tilde}  yields
    \[
        \|T\|^2 = \lambda_0 \le \|M_t\|^2 \le \|T\|^2.
    \]
    Hence, equality holds throughout, implying $\|M_t\| = \|T\|$; that is, $t \in \Omega$. By hypothesis, $\Omega = \{t_0\}$, so $t = t_0$. Conversely, since $\lambda_0 = \|T\|^2 = \|M_{t_0}\|^2 = \|M_{t_0}^*M_{t_0}\|$, the positivity of $M_{t_0}^*M_{t_0}$ ensures $\lambda_0$ is its largest eigenvalue, so $\det(\lambda_0 I_2 - M_{t_0}^*M_{t_0}) = 0$. In view of \eqref{defn of Phi}, this means $\Phi(t_0) = 0$.

    We have thus shown that $\mathcal{N}(\Phi) \cap \sigma(A) = \{t_0\}$. By Lemma~\ref{lem:kernel_equality}, it follows that $\mathcal{N}\big(\Phi(A)\big) = \mathcal{N}(A - t_0 I_H)$. Consequently, $Au = t_0 u$ and $Av = t_0 v$. As $(u, v)^\top \neq 0$, the vector $u$ or $v$ is non-zero, proving that $t_0$ is an eigenvalue of $A$.
\end{proof}

\begin{lemma}\label{from A to T} Let $\varphi \colon \mathbb{R}_+ \to \mathbb{R}$ be continuous and satisfy $\varphi(t) \geq \varphi(0) \geq 0$ for all $t \in \mathbb{R}_+$.
    Given $A \in \mathbb{B}(K, H)$ and $a, b, c \in \mathbb{C}$, let $M_t$, $T$  and $\Omega$ be as defined in \eqref{defn of M t}, \eqref{def of T}  and \eqref{defn of Omega}, respectively.
    Suppose there exists $t_0\in \Omega\cap \sigma_p(|A^*|)$. Then $T$ attains its norm.
\end{lemma}
\begin{proof}From Lemma~\ref{lem:norm_equivalence}, it suffices to verify that $\widetilde{T}$ attains its norm.  By assumption, $\|M_{t_0}\|=\|T\|$ and there exists a unit vector $v \in H$ such that $|A^*|v = t_0v$, which guarantees
    $$g(|A^*|)v=g(t_0)v,\quad \forall g\in C\big(\sigma(|A^*|)\big).$$
 Since the matrix $M_{\scriptscriptstyle t_0}$ acts on the finite-dimensional space $\mathbb{C}^2$, it  evidently attains its norm. Let $\xi=(\alpha,\beta)^T\in \mathbb{C}^2$ be a unit vector such that $\left\| M_{\scriptscriptstyle t_0} \xi \right\| = \|T\|$.

Define the unit vector $x = (\alpha v,\beta v)^T\in H \oplus H$. Due to $\varphi(|A^*|)v = \varphi(t_0)v$, we compute
    \begin{align*}
       \widetilde{T}x = \begin{pmatrix}
            a\alpha v + t_0\beta v \\
            c t_0\alpha v + b\varphi(t_0)\beta v
        \end{pmatrix}.
    \end{align*}
    Taking the squared norm and using  $\|v\|=1$  yields
    \begin{align*}
        \big\|\widetilde{T}x\big\|^2 &= \Big| a\alpha + t_0\beta\Big|^2 + \Big|c t_0\alpha + b \varphi(t_0)\beta\Big|^2\\
         &=\left\| M_{\scriptscriptstyle t_0} \xi \right\|^2 = \|T\|^2=\big\|\widetilde{T}\big\|^2,
    \end{align*}
     confirming that $\widetilde{T}$ attains its norm.
    \end{proof}

We are now ready to present the main result of this section.

\begin{theorem}\label{thm:norm_attainment_equivalence}
    Let $\varphi \colon \mathbb{R}_+ \to \mathbb{R}$ be a continuous function satisfying $\varphi(t) \geq \varphi(0) \geq 0$ for all $t \in \mathbb{R}_+$.
    Given $A \in \mathbb{B}(K, H)$ and $a, b, c \in \mathbb{C}$, let $M_t$ and $T$ be defined as in \eqref{defn of M t} and \eqref{def of T}, respectively.
    If the mapping $t \mapsto \|M_t\|$ is strictly increasing on $\mathbb{R}_+$, then the following statements are equivalent:
    \begin{enumerate}
        \item[\upshape (i)] $A$ attains its norm;
        \item[\upshape (ii)] $T$ attains its norm.
    \end{enumerate}
\end{theorem}

\begin{proof}
    Let $\Omega$ be as defined in \eqref{defn of Omega}.
    By assumption, $\Omega$ is the  singleton $\{\|A\|\}$.
    If $A$ attains its norm, then $\|A\| \in \Omega \cap \sigma_p(|A^*|)$.
    Hence, by Lemma~\ref{from A to T}, $T$ attains its norm.
    Conversely, suppose $T$ attains its norm.
    Then Lemma~\ref{lem:norm_attainment_necessity} yields $\|A\| \in \sigma_p(|A^*|)$, i.e., $A$ attains its norm.
\end{proof}

As immediate consequences of Theorem~\ref{thm:norm increasing-01}, Corollary~\ref{cor:sufficient condition increasing norm} and Theorem~\ref{thm:norm_attainment_equivalence}, we obtain the following two corollaries.

\begin{corollary}\label{cor:attainment_derivative_condition}
    Let $\varphi \colon \mathbb{R}_+ \to \mathbb{R}$ be a continuous, nonnegative, and strictly increasing function that is differentiable on $(0, +\infty)$. If $\varphi(t) \geq \frac{1}{4} t\varphi^\prime(t)$ for all $t \in (0, +\infty)$, then for any $A \in \mathbb{B}(K, H)$ and $a, b, c \in \mathbb{C}$, the block operator matrix $T$ defined in \eqref{def of T} attains its norm if and only if $A$ attains its norm.
\end{corollary}

\begin{corollary}\label{cor:attainment_real_part_condition}
    Let $\varphi \colon \mathbb{R}_+ \to \mathbb{R}$ be a continuous, nonnegative, and strictly increasing function. If $a, b, c \in \mathbb{C}$ satisfy $\mathrm{Re}\bigl(\bar{a}\bar{b}c\bigr) \ge 0$, then for any $A \in \mathbb{B}(K, H)$, the block operator matrix $T$ defined in \eqref{def of T} attains its norm if and only if $A$ attains its norm.
\end{corollary}

We conclude this section with several illustrative examples.

\begin{example}
For $A \in \mathbb{B}(K, H)$, $a, b, c \in \mathbb{C}$, $k, d \in \mathbb{R}_+$, and $\alpha \in [0, 4]$, define $T$ as
\begin{equation*}
T = \begin{pmatrix}
    aI_{H} & A \\
    cA^* & b\big(kI_H + d|A|^\alpha\big)
\end{pmatrix}.
\end{equation*}
From the structure of $T$, the function $\varphi$ in \eqref{def of T} corresponds to $\varphi(t) = k + dt^\alpha$ for $t \in \mathbb{R}_+$.
By Corollary~\ref{cor:norm increasing-01} and Theorem~\ref{thm:norm_attainment_equivalence}, $T$ attains its norm if and only if $A$ does.
Note that the degenerate case where $k = \alpha = 0$ and $d = 1$ was previously treated in \cite[Theorem~2.5]{WX}.
\end{example}

\begin{example}
For $d > 0$, let $H = L^2([0, d])$ be the Hilbert space of square-integrable functions on $[0, d]$.
Define $A \in \mathbb{B}(H)_+$ by $A = M_t$, the multiplication operator by the independent variable. That is,
\[
(Af)(t) = tf(t), \quad \forall f \in H,\; \forall t \in [0, d].
\]
It is evident that $\sigma(A) = [0, d]$ and $\sigma_p(A) = \emptyset$; consequently, $A$ fails to attain its norm.

Now, let $A$ be given as above.
For $a, b, c \in \mathbb{C}$ and a continuous function $\varphi: \mathbb{R}_+ \to \mathbb{R}$ with $\varphi(t) \geq \varphi(0) \geq 0$,
define $T\in\mathbb{B}(H\oplus H)$ via \eqref{def of T}. If the set $\Omega$ defined by \eqref{defn of Omega} is a singleton, then
Lemma~\ref{lem:norm_attainment_necessity} implies that $T$ also fails to attain its norm.
\end{example}

\begin{example}\label{ex:counterexample_necessity}
 Let $L^2_a(\mathbb{D})$ be the Bergman space consisting of analytic functions in $L^2(\mathbb{D}, \mathrm{d}A)$, where $\mathbb{D}$
    is the open unit disk in $\mathbb{C}$, and $dA(z)$ is the area measure on $\mathbb{D}$ normalized so that the area of $\mathbb{D}$ is 1.
     Let $P \colon L^2(\mathbb{D}, \mathrm{d}A) \to L^2_a(\mathbb{D})$ denote the Bergman projection.
 It is well-known that the sequence $\{e_n\}_{n=0}^\infty$, with $e_n(z) = \sqrt{n+1}\,z^n$, forms an orthonormal basis for $L^2_a(\mathbb{D})$.
For a symbol $\psi \in L^\infty(\mathbb{D})$, the Toeplitz operator $T_\psi$ on $L^2_a(\mathbb{D})$ is defined by
\[
T_\psi g = P(\psi g), \qquad g \in L^2_a(\mathbb{D}).
\]
If $\psi$ is radical (i.e., $\psi(z) = \psi(|z|)$), then $T_\psi$ is diagonalized by $\{e_n\}$, with eigenvalues
\[
\lambda_n = \langle T_\psi e_n, e_n \rangle = 2(n+1) \int_0^1 \psi(r)\, r^{2n+1} \,\mathrm{d}r,  \quad n\in\{0\}\cup\mathbb{N}.
\]
We refer to \cite[Theorem~1]{Xu1} for details.

Now take $\psi_0(z) = \bar{z}$ and let $A = T_{\psi_0}$. Then $(A^*f)(z) = z f(z)$ for all $f \in L^2_a(\mathbb{D})$ and $z \in \mathbb{D}$, so that $A A^* = T_{|z|^2}$. Consequently,
\[
A A^* = \operatorname{diag}(\lambda_0, \lambda_1, \dots), \qquad \text{where } \lambda_n = \frac{n+1}{n+2},\;  n\in\{0\}\cup\mathbb{N}.
\]
Hence,
\[
\sigma_p(|A^*|) = \bigl\{\sqrt{\lambda_n} : n\in \{0\}\cup\mathbb{N}\bigr\}
\quad\text{and}\quad
\sigma(|A^*|) = \sigma_p(|A^*|) \cup \{1\}.
\]
Since $\|A\|=1$ and $1\notin \sigma_p(|A^*|)$, $A$ does \emph{not} attain its norm.

Next, set $H = L^2_a(\mathbb{D})$ and define $T \in \mathbb{B}(H \oplus H)$ by
\[
T =
\begin{pmatrix}
-2 I_H & A \\
A^* & 2 |A|^5
\end{pmatrix}.
\]
In terms of the notation from \eqref{def of T}, this corresponds to
\[
a = -2,\quad b = 2,\quad c = 1,\quad \varphi(t) = t^5.
\]
By Example~\ref{ex:01}, there exists $n_0 \in \mathbb{N}$ such that $\sqrt{\lambda_{n_0}} \in \Omega \cap \sigma_p(|A^*|)$. An application of Lemma~\ref{from A to T} now shows that $T$ \emph{does} attain its norm.
\end{example}

\begin{example}\label{ex:counterexample_necessity_2}
Let $t_*$ be the unique root of $h(t)=0$ in $(0,1)$, where $h$ is defined by \eqref{defn of h}. Fix two numbers $t_1, t_2 \in (t_*, 1)$ with $t_1 < t_2$, and set $E = [0, t_1] \cup [t_2, 1]$. Let $H = L^2(E)$ and define $A \in \mathbb{B}(H)_+$ by
\[
(Ax)(t) = \rho(t)x(t), \quad \forall x \in H,\; t \in E,
\]
where
\[
\rho(t) =
\begin{cases}
t, & t \in [0, t_1], \\
1, & t \in [t_2, 1].
\end{cases}
\]
It follows that $\sigma(A) = [0, t_1] \cup \{1\}$ and $\sigma_p(A) = \{1\}$. In particular, $A$ attains its norm, but $t_* \notin \sigma_p(A)$.

For $t \in \mathbb{R}_+$, let $M_t$ be as defined in \eqref{M t 4 counterexample}, and define $T \in \mathbb{B}(H \oplus H)$ by
\[
T =
\begin{pmatrix}
-2 I_H & A \\
A & 2 A^5
\end{pmatrix}.
\]
By Example~\ref{ex:01}, the set $\Omega$ from \eqref{defn of Omega} is the singleton $\{t_*\}$. If $T$ were to attain its norm, then Lemma~\ref{lem:norm_attainment_necessity} would imply $t_* \in \sigma_p(A)$, contradicting the construction above. Hence, $T$ does \emph{not} attain its norm.
\end{example}

\end{document}